\documentclass[12pt,a4paper]{amsart}
 \pdfoutput=1
 
 \usepackage{amsmath,amssymb,amsthm}
 \usepackage{bm}  
 \usepackage{mathtools}
\usepackage{tikz}
\usetikzlibrary{arrows}
 \usepackage{thmtools}
 \usepackage{float} 
 
 \usepackage{xcolor} 	
 \usepackage{hyperref}
 \hypersetup{
 	colorlinks,
     linkcolor={red!60!black},
     citecolor={green!60!black},
     urlcolor={blue!60!black},
 }
 \usepackage[maxbibnames=99]{biblatex}
 \usepackage[utf8]{inputenc}
  \usepackage{csquotes}
 \usepackage[T1]{fontenc}
 \usepackage{lmodern}
 \usepackage[babel]{microtype}
 \usepackage[english]{babel}
 
 \usepackage{geometry}
 \usepackage{enumitem}

\DeclareMathAlphabet{\mymathbb}{U}{BOONDOX-ds}{m}{n}

\theoremstyle{plain}
\newtheorem{theorem}{Theorem}[section]

\newtheorem{proposition}[theorem]{Proposition}

\newtheorem{claim}[theorem]{Claim}
\newtheorem{corollary}[theorem]{Corollary}
\newtheorem{lemma}[theorem]{Lemma}

\newtheorem{observation}[theorem]{Observation}

\newtheorem{question}[theorem]{Question}
\newtheorem{conjecture}[theorem]{Conjecture}

\newtheorem*{theorem*}{Theorem}
\newtheorem*{corollary*}{Corollary}

\theoremstyle{definition}

\title{Large flames in rooted acyclic digraphs without backward-infinite paths}

\author{Attila Jo\'{o}}
\author{Qiuzhenyu Tao}
\address{Attila Jo\'{o}, Department of Mathematics, Technion, Haifa, Israel 32000.}
\address{Qiuzhenyu Tao, Universit{\"a}t Hamburg, Department of Mathematics, Bundesstra{\ss}e~55 (Geomatikum), 20146~Hamburg, Germany.}
\email{a.joo@technion.ac.il; qiuzhenyu.tao@uni-hamburg.de}

\begin{document}
\begin{abstract}
An $r$-rooted digraph is a flame if for each non-root vertex $v$, there is a set of edge-disjoint directed paths from $r$ to $v$ that covers all ingoing 
edges of $v$. The study of flames was initiated by Lovász, who showed that in a finite rooted digraph, the edge-minimal subgraphs that preserve all 
local edge-connectivities from the root are always flames. It is known that the edge sets of the flame subgraphs of any finite rooted digraph form a 
greedoid. Szeszlér showed recently that if the digraph is acyclic, then the bases of this greedoid are the bases of a matroid. 
We show that a suitable formulation of Szeszlér's theorem is valid for infinite digraphs under the additional assumption that there are no 
backward-infinite directed paths (which assumption is indeed essential). We also prove that the ``correct'' infinite generalisation of Lovász's theorem 
also holds for this class of digraphs.
\end{abstract}
\maketitle

\section{Introduction}
An $r$-rooted digraph is a digraph $D=(V,E)$ (where parallel edges are allowed, but loops are not) with a prescribed root vertex $r$ that has no 
ingoing edges. Assume that $D$ is finite. We seek a smallest possible set $L\subseteq E$ that preserves all local edge-connectivities from the root, 
i.e., $\lambda_D(r,v)=\lambda_{D(L)}(r,v)$ for each $v\in V\setminus \{ r \}$. Here, $D(L):=(V,L)$ and $\lambda_D(r,v)$ denotes the local 
edge-connectivity from $r$ to $v$, i.e., the maximal number of edge-disjoint $rv$-paths in $D$. 
Clearly, for each $v\in V\setminus \{ r \}$, such an $L$ contains at least $\lambda_D(r,v)$ ingoing edges of $v$, which implies the lower bound 
$\left|L \right| \geq \sum_{v\in V \setminus \{ r 
\}}\lambda_D(r,v) $. Lovász showed that, perhaps surprisingly, this lower bound is always sharp (see \cite[Theorem 2]{lovasz1973connectivity}). 
To recall several related results, we now introduce some notations.

By a ``path'', we always mean a directed path, and we identify paths with their edge sets. Let $D=(V,E)$ be an $r$-rooted digraph. For $v\in V\setminus \{ r \}$, we donte by $\delta_{D}(v)$ the set 
of ingoing edges of $v$, and we define $\mathcal{G}_D(v)$ to be the collection of sets $I\subseteq \delta_{D}(v)$, for which there is a set of 
edge-disjoint $rv$-paths (i.e. paths from $r$ to $v$) in $D$ that covers $I$. Furthermore, let 
\[ \mathcal{G}(D):=  \left\lbrace \bigcup_{v\in V\setminus \{ r \}}I_v:\ (\forall v\in V\setminus \{ r \})(I_v\in \mathcal{G}_D(v))\right\rbrace. \] 
We call $D=(V,E)$ a \emph{flame} if 
$E\in \mathcal{G}(D)$. It is worth mentioning that if $D$ is finite, then
$M_{v,D}:=(\delta_{D}(v),\mathcal{G}_D(v))$ is a gammoid\footnote{In fact, gammoids are exactly those matroids that can be represented this 
way.}. 
Therefore, $M_D:=(E,\mathcal{G}(D))$ is the direct sum of these gammoids and hence a gammoid itself that we call the \emph{edge-connectivity 
gammoid} of $D$. 

Fix a (possibly infinite) $r$-rooted digraph 
$D=(V,E)$. A set $F\subseteq E$ is called a flame if the $r$-rooted spanning subdigraph $D(F):=(V,F)$ is a flame. Note that while every flame is an element of 
$\mathcal{G}(D)$, the converse does not hold in general.
\begin{theorem}[{Joó, \cite{jooGreedoidFlame2021}}]\label{thm: flame greedoid}
 In any finite $r$-rooted digraph $D$, the flames in $D$ form a greedoid.\footnote{In other words, if $F$ and $F'$ are flames in $D$ with $\left|F 
 \right|<\left|F' \right|$, then there is an $e\in F' \setminus F$ such that $F\cup \{ e \}$ is a flame.}
 \end{theorem}   
 A set $L\subseteq E$ is called \emph{large at} $v$ if $D(L)$ admits a set $\mathcal{P}_v$ of edge-disjoint $rv$-paths such that $\mathcal{P}_v$ 
 has a transversal that is an $rv$-cut\footnote{An $rv$-cut is an edge set $C$ such that every $rv$-path in $D$ meets $C$.} in $D$. Note that, if 
 $\lambda_D(r,v)$ is finite, then, by Menger's theorem, $L$ being large at $v$ is equivalent to $\lambda_{D(L)}(r,v)=\lambda_D(r,v)$. However, 
 in general, being large at $v$ is a stronger property than merely preserving the local edge-connectivity from $r$ to $v$. A set $L\subseteq E$ is 
\emph{large} if it is large at every $v\in V\setminus \{ r \}$. Lovász's theorem can be phrased in the following way:

\begin{theorem}[{Lovász, \cite[Theorem 2]{lovasz1973connectivity}}]\label{thm: lovasz orig}
 Every finite rooted digraph  admits a large flame. 
\end{theorem}
Observe that if $D$ is finite, then every large $L$ must include a base of the edge-connectivity gammoid $M_D$.  Lovász's theorem states that $M_D$ 
always admits a large base. Szeszlér showed recently that if the digraph is acyclic, then all the bases of $M_D$ are large:
\begin{theorem}[{Szeszlér, \cite[Corollary 1]{szeszler2025some}}]\label{thm: orig Szesz}
 If  $D=(V,E)$ is  a finite acyclic $r$-rooted digraph, then every maximal element of $\mathcal{G}(D)$ is large. 
\end{theorem}
\noindent The following example shows that Szeszlér's theorem may fail for infinite digraphs:
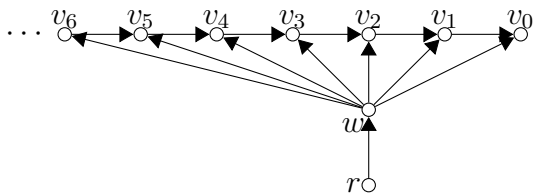
\begin{figure}[H]
\centering

\begin{tikzpicture}
\node[circle,inner sep=0pt,draw,minimum size=5]  (v9) at (0,0) {};
\node[circle,inner sep=0pt,draw,minimum size=5]  (v8) at (0,1) {};

\node[circle,inner sep=0pt,draw,minimum size=5]  (v7) at (2,2) {};
\node[circle,inner sep=0pt,draw,minimum size=5]  (v6) at (1,2) {};
\node[circle,inner sep=0pt,draw,minimum size=5]  (v5) at (0,2) {};
\node[circle,inner sep=0pt,draw,minimum size=5]  (v4) at (-1,2) {};
\node[circle,inner sep=0pt,draw,minimum size=5]  (v3) at (-2,2) {};
\node[circle,inner sep=0pt,draw,minimum size=5]  (v2) at (-3,2) {};
\node[circle,inner sep=0pt,draw,minimum size=5]  (v1) at (-4,2) {};
\node  at (-4.5,2) {$\cdots$};

\draw[-triangle 60]  (v1) edge (v2);
\draw[-triangle 60]  (v2) edge (v3);
\draw[-triangle 60]  (v3) edge (v4);
\draw[-triangle 60]  (v4) edge (v5);
\draw[-triangle 60]  (v5) edge (v6);
\draw[-triangle 60]  (v6) edge (v7);
\draw[-triangle 60]  (v8) edge (v7);
\draw[-triangle 60]  (v8) edge (v6);
\draw[-triangle 60]  (v8) edge (v5);
\draw[-triangle 60]  (v8) edge (v4);
\draw[-triangle 60]  (v8) edge (v3);
\draw[-triangle 60]  (v8) edge (v2);
\draw[-triangle 60]  (v8) edge (v1);
\draw[-triangle 60]  (v9) edge (v8);
\node at (-0.2,0) {$r$};
\node at (-0.2,0.8) {$w$};
\node at (2,2.2) {$v_0$};
\node at (1,2.2) {$v_1$};
\node at (0,2.2) {$v_2$};
\node at (-1,2.2) {$v_3$};
\node at (-2,2.2) {$v_4$};
\node at (-3,2.2) {$v_5$};
\node at (-4,2.2) {$v_6$};
\end{tikzpicture}
\caption{The horizontal edges together with the edge from  $r$ to $w$ form a maximal element $B$ of $\mathcal{G}(D)$. Clearly, $B$ is not large because 
$\lambda_D(r,v_n)=1$ but $\lambda_{D(B)}(r,v_n)=0$ for each $n\in \mathbb{N}$.} \label{fig: backinf}
\end{figure}
\noindent We say that an infinite path is \emph{backward-infinite} if it has a final vertex but no starting vertex (like the path consisting of the 
horizontal edges in Figure \ref{fig: backinf}). The first result of this paper states that such paths are the only additional obstacles corresponding to 
Szeszlér's theorem:  
  
\begin{restatable*}{theorem}{infSzesz}\label{thm: main Szesz}
 If  $D=(V,E)$ is  an  acyclic $r$-rooted digraph without backward-infinite paths, then every maximal element of $\mathcal{G}(D)$ is  large.
\end{restatable*}
\noindent Szeszlér's proof of Theorem \ref{thm: orig Szesz} applies induction on the number of vertices and uses the fact that a finite acyclic 
digraph has a sink. The latter fails for infinite digraphs, for example, consider a forward-infinite path. Moreover, induction on the number of vertices 
is not applicable if the vertex set is infinite. Therefore, proving Theorem \ref{thm: main Szesz} requires completely new tools compared to 
Szeszlér's approach. 
 
It is possible to define the vertex version of flames and largness, where internally vertex-disjoint paths are considered instead of edge-disjoint ones. 
There are several positive results about the existence of large flames in infinite digraphs for the vertex version (see \cite{joo2019vertex, 
erde2021enlarging, jacobs2023lovasz, gut2024large}), but none had been established previously for the edge version.
The second result of the present paper establishes the existence of large flames in acyclic digraphs without backward-infinite paths:

\begin{restatable*}{theorem}{largeflame}\label{thm: main large flame acyc}
 In every $r$-rooted acyclic digraph $D=(V,E)$ that does not contain backward-infinite paths,  every flame of $D$ extends to a large flame of $D$. 
\end{restatable*}
\noindent Since the empty set is trivially a flame in $D=(V,E)$, Theorem \ref{thm: main large flame acyc} in particular guarantees the existence of 
a large flame.

\section{Proof of the main results}
We will use the following theorems:
\begin{theorem}[Aharoni and Berger, \cite{aharoni2009menger}]\label{thm: gen Menger}
 If $D=(V,E)$ is a (possibly infinite) digraph and $s,t\in V$ are distinct, then there is a set $\mathcal{P}$ of edge-disjoint $st$-paths and an $st$-cut 
 $C$ such that 
 $C$ is a transversal of $\mathcal{P}$ (i.e. $C$ contains exactly one edge from each path $P\in \mathcal{P}$). 
\end{theorem}
For a set $\mathcal{P}$ of edge-disjoint finite paths, let $\mathsf{in}(\mathcal{P})$ and $\mathsf{ter}(\mathcal{P})$ denote the sets of initial and terminal edges, respectively, of the paths in $\mathcal{P}$ .
\begin{theorem}[Diestel and Thomassen, \cite{diestel2006cantor}]\label{thm: Pym}
 Assume that $D=(V,E)$ is a (possibly infinite) digraph,  $s,t\in V$ are distinct, and $\mathcal{P}$ and $\mathcal{Q}$ are sets of edge-disjoint 
 $st$-paths in $D$. Then there exists a set $\mathcal{R}$ of $st$-paths in $D$ such that 
 $\mathsf{in}(\mathcal{P})\subseteq\mathsf{in}(\mathcal{R})$ and $\mathsf{ter}(\mathcal{Q})\subseteq \mathsf{ter}(\mathcal{R})$.
\end{theorem}

For a vertex set $X$, let $\delta_D(X)$ be the set of edges entering $X$. Fix an $r$-rooted (possibly infinite) digraph $D=(V,E)$. A set  $X\subseteq V\setminus \{ r \} $ is $v$\emph{-linked} if $v\in 
X$ and there exists a set $\{ P_e:\  e\in \delta_D(X) \}$ of 
edge-disjoint paths, where path $P_e$ has $e$ as its initial edge and $v$ as its terminal vertex.

\begin{lemma}\label{lem: two bubble union}
 If $X$ is $v$-linked, $w\in X$, and $Y$ is $w$-linked, then $X\cup Y$ is $v$-linked. 
\end{lemma}
\begin{proof}
    Let $\{P_e:\  e\in \delta_D(X)\}$ and $\{Q_f:\  f\in \delta_D(Y)\}$ witness that $X$ is $v$-linked and $Y$ is  $w$-linked, respectively. We construct a path system $\{R_g:\  g\in \delta_D(X\cup Y)\}$ to witness that $X\cup Y$ is also $v$-linked. For any $g\in \delta_D(X\cup Y) \cap  \delta_D(X) $, set $R_g:= P_g$. If $g\in \delta_D(X\cup Y) \setminus  \delta_D(X) $, then $g\in \delta_D(Y) $. Since the path $Q_g$ terminates at $w\in X$, it must enter $X$. Let $e$ be the first edge of $Q_g$ in $\delta_D(X)$, and let $Q_g'$ be the initial segment of $Q_g$ up to $e$. We then define $R_g:=Q'_g \cup P_e$. It is routine to check that $\{R_g:\  g\in \delta_D(X\cup Y)\}$ is as desired.
\end{proof}

\begin{lemma}\label{lem: chain bubble union}
 If $\kappa$ is an infinite cardinal and $\left\langle X_\alpha:\ \kappa \right\rangle $ is an $\subseteq$-increasing continuous sequence of 
 $v$-linked sets, then $\bigcup_{\alpha<\kappa}X_\alpha$ is $v$-linked.
\end{lemma}
\begin{proof}
Set $X_\kappa:=\bigcup_{\alpha<\kappa}X_\alpha$. For each $\alpha \leq\kappa$, we construct by transfinite recursion a path system 
$\mathcal{P}^{\alpha}=\{ P^{\alpha}_e:\ e\in \delta_D(X_\alpha) 
\}$ witnessing that $X_\alpha$ is $v$-linked. We maintain the additional property that for every $\beta<\alpha$, each   $P_e^{\alpha}\in \mathcal{P}^{\alpha}$ has exactly one edge $f$ in 
$ \delta_D(X_\beta)$, and the terminal segment of $P_e^{\alpha}$ starting from $f$ is $P^{\beta}_f$. 

The set $X_0$ is $v$-linked by definition, and we let $\mathcal{P}^{0}$ be an arbitrary witness to this. 
Assume that $\mathcal{P}^{\beta}$ is already defined for $\beta<\alpha$, where $0<\alpha\leq\kappa$. Assume first that $\alpha=\beta+1$ is a successor ordinal. Let $\mathcal{Q}^{\alpha}=\{ Q^{\alpha}_f:\ f \in \delta_D(X_\alpha) \}$ be a path system witnessing that $X_\alpha$ is $v$-linked. 
For each $f \in \delta_D(X_\alpha)$, define $P^{\alpha}_f$ as the union of the initial segment of $Q^{\alpha}_f$ up to its first edge $e$ that enters $\delta(X_\beta)$ and the path $P^{\beta}_e$. Clearly, $ \mathcal{P}^{\alpha}:=\{ P^{\alpha}_f:\ f\in \delta_D(X_\alpha) \}$ shows that $X_\alpha$ is $v$-linked and maintains
 the desired property. 
Now assume that $\alpha$ is a limit ordinal. The sequence $\left\langle X_\alpha:\ \kappa \right\rangle $ is 
continuous by assumption, therefore $X_\alpha=\bigcup_{\beta<\alpha}X_\beta$. Then for each $e\in \delta_D(X_\alpha)$, there is a smallest $\beta<\alpha$ with $e\in \delta_D(X_\beta)$. We must take $P^{\alpha}_e:=P^{\beta}_e$. It is routine to check that 
$\mathcal{P}^{\alpha}:=\{ P^{\alpha}_e:\ e\in \delta_D(X_\alpha) 
\}$ is suitable. Finally, $\mathcal{P}^{\kappa}$ witnesses that $X_\kappa$ is $v$-linked.
\end{proof}

\begin{corollary}\label{cor: largest bubble}
 For every $v\in V\setminus \{ r \}$, there exists a $\subseteq$-largest $v$-linked set $X_{v,D}$ in $D$. Furthermore, for every $w\in X_{v,D}$, we have $X_{w,D}\subseteq X_{v,D}$. 
\end{corollary}
\begin{proof}
  Lemma \ref{lem: chain bubble union} ensures that the union $X_{v,D}$ of all $v$-linked sets is $v$-linked, and thus $X_{v,D}$ is the $\subseteq$-largest $v$-linked set. Now, for any $w\in X_{v,D}$, the set $X_{v,D}\cup X_{w,D}$  is 
  $v$-linked by Lemma \ref{lem: two bubble union}. By the maximality of $X_{v,D}$, it follows that $X_{w,D} \subseteq X_{v,D}$.
\end{proof}

\begin{proposition}\label{claim: vlinked covers}
 For a $v$-linked set $X$ and $I\in \mathcal{G}_D(v)$, there exists a path system $\mathcal{R}$ that covers $I$ while also witnessing that $X$ is $v$-linked. 
\end{proposition}
\begin{proof}
 Take a path system $\mathcal{P}$ witnessing that $X$ is $v$-linked and a path system $\mathcal{Q}$ witnessing $I\in \mathcal{G}_D(v)$. Apply Theorem \ref{thm: Pym} to $\mathcal{P}$ and the terminal segments of the paths in $\mathcal{Q}$ from their last edge in $\delta_D(X)$ in the digraph that we obtain from $D$ by identifying the vertices in $V\setminus X$.
\end{proof}

A set $X\subseteq V\setminus \{ r \}$ is \emph{fillable} in $D$ if there is a set $\mathcal{P}$ of edge-disjoint $rX$-paths with 
$\mathsf{ter}(\mathcal{P})=\delta_D(X)$. 
\begin{observation}\label{obs: fillable and bubble}
If $X$ is fillable and $v$-linked, then there is a set $\mathcal{P}$ of edge-disjoint $rv$-paths such that the $rv$-cut 
$C:=\delta_D(X)$ is a transversal of $\mathcal{P}$.  
\end{observation}
\begin{proof}
 Take a path system witnessing that $X$ is fillable, and extend these paths forward to reach $v$ via a path system witnessing that $X$ is $v$-linked. 
\end{proof}

\begin{lemma}\label{lem: superfillable}
 For every $v\in V\setminus \{ r \}$, the set $X_{v,D}$ is fillable. Moreover, if $D'$ is obtained from $D$ by adding a new edge $e$ that enters 
 $X_{v,D}$, 
 then 
 $X_{v,D}$ is fillable with respect to $D'$ as well.
\end{lemma}
\begin{proof}
Apply Theorem \ref{thm: gen Menger} to the digraph $D_t$ obtained from $D$ by contracting the vertex set $ X_{v,D}$ to a single vertex $t$, with $s:=r$ and $t$. Let $\mathcal{P}$ be the resulting path system and let $C$ be an $st$-cut in $D_t$ that forms a transversal of $\mathcal{P}$. Define $X$ as the set of vertices in $V$ that are not reachable from $r$ in $(V,E\setminus C)$. Clearly, $X \supseteq 
X_{v,D}$ and $\delta_D(X)=C$. 
We claim that $X$ is $v$-linked in $D$. Indeed, follow the terminal segments of the paths in $\mathcal{P}$ from $C$ until they first enter $X_{v,D}$, and then extend them using the fact that $X_{v,D}$ is $v$-linked. The resulting path system witnesses that $X$ is $v$-linked. By the maximality of $X_{v,D}$, it follows that $X=X_{v,D}$. Then the initial segments of the paths in $\mathcal{P}$ up to $C$ provide a path system that witnesses the fillability of $ X_{v,D}$.

We now turn to the proof of the ``moreover'' part. Let $\mathcal{P}$ be a path system witnessing that $X_{v,D} $ is fillable in $D$. Consider the digraph $D''$ obtained from $D'$ by reversing the orientation of all edges in $\bigcup \mathcal{P}$. It suffices to show that there is a path $P$ from $r$ to $X_{v,D}$ in $D''$, since taking the symmetric difference of $\bigcup \mathcal{P}$ and $P$ leads to a path system witnessing that $X_{v,D}$ is fillable with respect to $D'$.
Suppose for a contradiction that no such path exists. Let $X$ be the set of vertices in $V$ that are not reachable from $r$ after the deletion of $C$ from $D''$. Then $X \supsetneq X_{v,D}$, because $X$ must contain the tail of $e$. Since $X$ has no ingoing edges 
in $D''$, we conclude that  $\delta_D(X)$ is a transversal of $\mathcal{P}$. However, the terminal segments of the paths in $\mathcal{P}$ from $\delta_D(X)$ combined with the $v$-likedness of $X_{v,D}$ shows that $X$ is $v$-linked in $D$, which is a contradiction because $X \supsetneq X_{v,D}$.
\end{proof}

\begin{lemma}\label{lem: largness char}
 A set $L\subseteq E$ is large if and only if for every $v\in V \setminus \{ r \}$, the set $X_{v,D(L)}$ contains the tails of all the edges in
 $\delta_{D(E\setminus L)}(v)$.
\end{lemma}
\begin{proof}
    Let $L\subseteq E$ be large and let  $v\in V \setminus \{ r \}$. By definition, $D(L)$ admits a set $\mathcal{P}$ of edge-disjoint $rv$-paths with a transversal $C$ that is an $rv$-cut in $D$. Let $X$ be the set of vertices not reachable from $r$ in $(V,E\setminus C)$. The terminal segments of the paths in $\mathcal{P}$ from $C$ witness that $X$ is $v$-linked in $D(L)$, so $X_{v,D(L)} \supseteq X$. 
    Suppose, for a contradiction, that there is an edge $e\in \delta_{D(E\setminus L)}(v)$ that has its tail $u$ not in $ X$. Note that $e\notin C$ since $C\subseteq L$. By the definition of 
    $X$, there exists an $rv$-path $P_u$ in $(V,E\setminus C)$. Extending $P_u$ by $e$ then results in an $rv$-path in $(V,E\setminus C)$, contradicting the fact that $C$ is an $rv$-cut.
     This shows that the tail of every edge in $\delta_{D(E\setminus L)}(v)$ is in $X$, and hence, by $X \subseteq X_{v,D(L)}$, in $X_{v,D(L)}$ as well.

    For the other direction, suppose that for every vertex $v\in V \setminus \{ r \}$, the set $X_{v,D(L)}$ contains the tail of every edge in $\delta_{D(E\setminus L)}(v)$. Fix $v\in V \setminus \{ r \}$ and let $X:=X_{v,D(L)}$. By definition, $X$ is $v$-linked in $D(L)$, and by Lemma \ref{lem: superfillable}, $X$ is also fillable in $D(L)$.
    Observation \ref{obs: fillable and bubble} shows that there is a set $\mathcal{P}$ of edge-disjoint $rv$-paths in $D(L)$ such that
     $C:=\delta_{D(L)}(X)$ is a transversal of $\mathcal{P}$. We claim that $\delta_{D(L)}(X)=\delta_{D}(X)$, and therefore $C$ is also an $rv$-cut in $D$ and hence $\mathcal{P}$ witnesses the largeness of $L$ at $v$. Suppose for a contradiction that there exists an edge $e\in \delta_{D}(X)\setminus \delta_{D(L)}(X)$. Let $u$ and $w$ be the tail and head of $e$, respectively. Since $e\in \delta_{D(E\setminus L)}(w)$, the set $X_{w,D(L)}$ contains $u$ by assumption. Corollary \ref{cor: largest bubble} ensures that $ X_{w,D(L)}\subseteq X$, therefore $u\in X$ as well. But this contradicts $e\in \delta_{D}(X)$.
\end{proof}

Let $\mathsf{tail}(e)$ and $\mathsf{head}(e)$ denote the tail and head of an edge $e$ respectively.
\begin{observation}\label{obs: well-order}
 If $D=(V,E)$ is an a acyclic digraph without backward-infinite paths, then there is a well-order $<$ of $V$ such that 
 $\mathsf{tail}(e)<\mathsf{head}(e)$  for every 
 $e\in E$.
\end{observation}
\begin{proof}
 Any digraph that contains neither directed cycles nor backward-infinite paths always has a source vertex. To construct the suitable vertex ordering, apply tranfinite recursion: at each step, delete a source from the current digraph. This process terminates when all vertices are removed, and the order of deletion defines the desired ordering.
\end{proof}
\infSzesz
\begin{proof}
By Observation  \ref{obs: well-order}, we can fix a well-order $<$ on $V$ in which $\mathsf{tail}(e)<\mathsf{head}(e)$ for every $e\in E$. 
Let $L$ be a maximal element of $\mathcal{G}(D)$. Suppose for a contradiction 
 that $L$ is not large. It follows from Lemma \ref{lem: largness char}, that there is a $<$-smallest vertex $v\in V\setminus \{ r \}$ for which there is an edge $e \in \delta_{D(E\setminus L)}(v)$ such that $\mathsf{tail}(e)\notin X_{v,D(L)}$.   Lemma \ref{lem: superfillable} provides for $X:=X_{v,D(L)}$ a set $\mathcal{P}$ of edge-disjoint $rX$-paths in $D(L\cup \{ e \})$ such that  $\mathsf{ter}(\mathcal{P})=\delta_{D(L\cup \{ e \})}(X)$. Let $\mathcal{Q}$ be a path system witnessing $\delta_{D(L)}(v)\in \mathcal{G}_D(v)$. 
 \begin{claim}
  If $Q\in \mathcal{Q}$ and $f$ is the last edge of $Q$ that is in $\delta_D(X)$, then $f\in L$. 
 \end{claim}
 \begin{proof}
  Suppose for a contradiction that $f\in E\setminus L$, and let $w:=\mathsf{head}(f)$. The path $Q$ witnesses that $v$ is reachable by a directed path from $w$, thus we must have $w<v$. By the $<$-minimal choice of $v$, we know that $\mathsf{tail}(f)\in X_{w,D(L)}$.  Corollary \ref{cor: largest bubble} ensures  that $X_{w,D(L)}\subseteq X$. By combining these we conclude that $\mathsf{tail}(f)\in X$, which contradicts $f\in \delta_D(X)$.
 \end{proof}
 For each path $Q\in \mathcal{Q}$, let $f_Q$ be the last edge of $Q$ in $\delta_D(X)$, and let $Q'$ be the union of the terminal segment of $Q$ starting with $f_Q$ and the unique path in $\mathcal{P}$ whose last edge is $f_Q$. 
  By denoting $P_e$ the unique path in $\mathcal{P}$ with terminal edge $e$, the set $\mathcal{R}:=\{ Q':\ Q\in \mathcal{Q} \}\cup \{ P_e \}$ consists of pairwise edge-disjoint paths in $D$ witnessing 
 $\delta_{D(L)}(v)\cup \{ e \}\in \mathcal{G}_D(v)$. Therefore $L\cup \{ e \}\in \mathcal{G}(D)$, which contradicts the maximality of $L$.
 \end{proof}
We turn to the proof of the second  main result of the paper:
\largeflame
\begin{proof}
    By Observation \ref{obs: well-order}, we can fix a well-order $<$ on $V$ such that $\mathsf{tail}(e)<\mathsf{head}(e)$ for each $e\in E$. Let $\{ v_\alpha:\ \alpha<\alpha^{*} \}$ be the enumeration of $V\setminus \{ r \}$ according to $<$.
    Let $F$ be a flame in $D$. We apply transfinite recursion on $\alpha^{*}$. At step $\alpha<\alpha^{*} $, we delete some ingoing edges of $v_\alpha$ that are not in $F$. Let $L_\alpha \subseteq E$ be what we have left at the beginning of step $\alpha $. Let $\mathcal{P}_\alpha$ be  a system of edge-disjoint paths that covers $\delta_{D(F)}(v_\alpha)$  and witnesses that $X_\alpha:=X_{v_\alpha, D(L_\alpha)}$ is $v_\alpha$-linked in $D(L_\alpha)$ (such a $\mathcal{P}_\alpha$ exists by Proposition \ref{claim: vlinked covers}). We delete those ingoing edges of 
    $v_\alpha$ that are not used by $\mathcal{P}_\alpha$. 
 \begin{observation}\label{obs: deleted in Xalpha}
The tail of each deleted edge is in  $X_\alpha$.   
 \end{observation}
 \begin{proof}
By the definition of $v_\alpha$-linked, every edge in $D(L_\alpha)$ entering $X_\alpha$ must be the initial edge of a path in $\mathcal{P}_\alpha$. Consequently, no ingoing edge of $v_\alpha$ that enters $X_\alpha$ has been deleted.  
 \end{proof}
 \begin{observation}
  There is a set  $\mathcal{Q}_\alpha$ 
      of edge-disjoint $rv$-paths in $D(L_\alpha)$  such that $\mathcal{P}_\alpha$ consists of terminal segments 
      of the paths in $\mathcal{Q}_\alpha$.  In other words,  there are backward-coninuations of the 
              paths in $\mathcal{P}_\alpha$ resulting in a set $\mathcal{Q}_\alpha$
                  of edge-disjoint $rv$-paths in $D(L_\alpha)$.
 \end{observation}       
 \begin{proof}
  It follows directly from the first part of Lemma \ref{lem: superfillable}.
 \end{proof} 
 \begin{observation}\label{obs: no deletion}
 No path in $\mathcal{Q}_\alpha$ uses an ingoing edge of any vertex $v_\beta$ with $\beta>\alpha$. 
\end{observation}
\begin{proof}
 Each path in $\mathcal{Q}_\alpha$ goes from $r$ to $v_\alpha$. By the definition of $<$, the set $\{ r \}\cup \{ v_\gamma:\ \gamma \leq \alpha \}$ has no ingoing edges in $D$. Therefore, paths in $\mathcal{Q}_\alpha$ do not leave the set $\{ r \}\cup \{ v_\gamma:\ \gamma \leq \alpha \}$.
\end{proof}
 We show that $L:= \bigcap_{\alpha<\alpha^{*}}L_\alpha$  is a large flame. Observation 
 \ref{obs: no deletion} ensures that for every $\alpha<\alpha^{*}$, $\mathcal{Q}_\alpha$ is a path system in $D(L)$ (not just in $D(L_\alpha)$). 
 Furthermore,  it is clear from the construction that 
 $\delta_{D(L_{\alpha+1})}(v_\alpha)=\delta_{D(L)}(v_\alpha)=\mathsf{ter}(\mathcal{Q}_\alpha)$  for each  $\alpha<\alpha^{*}$. Therefore, $L$ is a flame. 
 
 We turn to the proof of the largeness of $L$. The set $X_\alpha:=X_{v_\alpha,D(L_{\alpha})}$ is  $v_\alpha$-linked in 
  $D(L)$, since it is $v_\alpha$-linked in $D(L_\alpha)$ witnessed by $\mathcal{P}_\alpha$ and Observation \ref{obs: no deletion} guarantees $\bigcup \mathcal{P}_\alpha \subseteq L$. Hence, $X_{v_\alpha,D(L)} \supseteq X_\alpha $ by the definition of $X_{v_\alpha,D(L)}$. Moreover, Observation \ref{obs: deleted in Xalpha} states that $X_\alpha$ contains the tail of every edge in $\delta_{D(E\setminus L)}(v_\alpha)$. Combining these, we conclude  that $X_{v_\alpha,D(L)} $ also contains the tail of every edge in $\delta_{D(E\setminus L)}(v_\alpha)$. Lemma \ref{lem: largness char} therefore guarantees that $L$ is large.
\end{proof}

\section{Open problems}
We expect that Theorem \ref{thm: main large flame acyc} holds without any restriction on the digraph:
\begin{conjecture}
 In every $r$-rooted digraph $D$, every flame extends to a large flame. In particular, every $r$-rooted digraph $D$ admits a large flame.
\end{conjecture}

For an infinite digraph $D$, the set $\mathcal{G}(D)$ may fail to have a maximal element. This motivates the question if it is 
possible to find an infinite generalisation of Theorem \ref{thm: orig Szesz} that does not rely on the existense of such a maximal element.  Our candidate is the following:
\begin{conjecture}\label{conj: Szesz}
 If  $D$ is an acyclic $r$-rooted digraph that does not contain backward-infinite paths, then every $I\in \mathcal{G}(D)$ is included in a flame of $D$. 
\end{conjecture}
 We claim that this is indeed a generalisation of Theorem \ref{thm: orig Szesz}. To see this, let $D$ be a finite acyclic $r$-rooted digraph and let $B$ be a maximal element of $\mathcal{G}(D)$. Since flames are elements of $\mathcal{G}(D)$, no proper superset of $B$ is a flame. Therefore, Conjecture \ref{conj: Szesz} implies that $B$ itself is a flame.  Since in $D(B)$, every $v\in V \setminus \{ r \}$ has $\lambda_D(r,v)$ ingoing edges, $B$ being a flame implies that $B$ is large.

\begin{question}
Is it true that  if $D$ is an acyclic $r$-rooted digraph without backward-infinite paths, then all the maximal elements of $\mathcal{G}(D)$ are flames?
\end{question}

\section*{Acknowledgements}
This work was completed during our visit to the Center for Combinatorics at Nankai University. We are grateful for the warm hospitality and the stimulating research environment provided by the host.

\printbibliography
\end{document}